\documentclass[11pt]{amsart}
\usepackage{latexsym}
\usepackage{amsfonts}
\usepackage{amsmath,amssymb}
\usepackage{color}

\newcommand{\field}[1]{\mathbb{#1}}
\newcommand{\C}{\field{C}}

\newcommand{\ignore}[1]{}

\newtheoremstyle{s2}{9pt}{9pt}{\rm}{}{\bf}{.}{0.5em}{}
\theoremstyle{s2}
\newtheorem{definition}{Definition}[section]

\newtheorem{remark}[definition]{Remark}

\newtheoremstyle{s1}{9pt}{9pt}{\it}{}{\bf}{.}{0.5em}{}
\theoremstyle{s1}
\newtheorem{lemma}[definition]{Lemma}
\newtheorem{theorem}[definition]{Theorem}

\newtheoremstyle{changednumber}{}{}{\itshape}{}{\bfseries}{.}{.5em}{#1 \thmnote{#3}}
\theoremstyle{changednumber}
\newtheorem*{changednumbertheorem}{Theorem}

\font\tenmsy=msbm10

\def\Bbb#1{\hbox{\tenmsy#1}}

\DeclareMathOperator{\rank}{rank}

\title[Finite $\mathcal{A}$-determinacy of generic homogeneous map germs in $\mathbb{C}^3$]{Finite $\mathcal{A}$-determinacy of generic homogeneous map germs in $\mathbb{C}^3$} \makeatletter

\@addtoreset{equation}{section}
\makeatother
\author{M. \ Farnik \& Z. Jelonek \& M.A.S. Ruas}
\address[M. Farnik]{Jagiellonian University\\
Faculty of Mathematics and Computer Science\\
{\L}ojasiewi\-cza~6, 30-348 Krak\'ow, Poland}
\email{michal.farnik@gmail.com}

\address[Z. Jelonek]{Instytut Matematyczny\\
Polska Akademia Nauk\\
\'Sniadeckich 8, 00-656 Warszawa, Poland}
\email{najelone@cyf-kr.edu.pl}

\address[M.A.S. Ruas]{Departamento de Matem\'atica,
ICMC-USP, Caixa Postal 668, 13560-970 S\~ao Carlos, S.P., Brasil}
\email{maasruas@icmc.usp.br}


\subjclass[2010]{14 R 99, 32 A 10}
\thanks{The  authors are partially supported by the grant of Narodowe Centrum Nauki, grant number 2015/17/B/ST1/02637, additionally the third author is partially supported by the FAPESP grant 2014/00304-2}

\begin{document}

\begin{abstract}
Denote by $H(d_1,d_2,d_3)$ the set of all homogeneous polynomial mappings $F=(f_1,f_2,f_3): \C^3\to\C^3$, such that $\deg f_i=d_i$. 
We show that  if  $\gcd(d_i,d_j)\leq 2$ for $1\leq i<j\leq 3$ and $\gcd(d_1,d_2,d_3)=1$, then there is a non-empty Zariski open subset $U\subset H(d_1,d_2,d_3)$ such that for every mapping $F\in U$ the map germ $(F,0)$ is $\mathcal{A}$-finitely determined. Moreover, in this case we compute the number of discrete singularities ($0$-stable singularities) of a generic mapping $(f_1,f_2,f_3):\C^3\to\C^3$, where $\deg f_i=d_i$.

\end{abstract}

\maketitle

\bibliographystyle{alpha}

\section{Introduction}
Let $\Omega(d_1,\ldots,d_n)$ denote the set of all polynomial mappings $F=(f_1,\ldots,f_n): \C^n\to\C^n$ such that  $\deg f_i=d_i$. 
We have proved in \cite{fjr} that there is an open subset $U\subset\Omega(d_1,\ldots,d_n)$ such that for every $F\in U$ the mapping  $F$ is transversal to the Thom-Boardman varieties and satisfies the normal crossings property. 
Moreover, by \cite{jel} all such mappings are topologically equivalent, in particular they have the same number of discrete singularities. If $U_0\subset \Omega(d_1,\ldots,d_n)$ is the maximal open subset
with these properties (i.e., for every $F\in U_0$ the mapping $F$ has constant topological type and it is transversal to the Thom-Boardman varieties and satisfies the normal crossings property) then we say that every mapping $F\in U_0$ is a generic mapping,

Let $F\in \Omega(d_1,\ldots,d_n)$ be a generic polynomial mapping.  In particular in Mathers nice dimensions (see \cite{mathVI}) $F$ is a stable mapping. 
In \cite{fjr} we have computed the number of cusps and nodes for $F$ in dimension $n=2$. Now we would like to compute the number of discrete singularities ($0$-stable singularities) in dimension $n=3$.

Note that a generic polynomial mapping $F: \C^n\to\C^n$ can be defined at infinity only if $d_1=d_2=\ldots=d_n=d$. However even in this case the mapping $F$ (if non-linear) has to be degenerate at infinity, i.e., the whole hyperplane at infinity is a 
 component of the critical set of $F.$  Indeed the topological degree of $F$ is $\mu(F)=d^n$, but the mapping $F$ restricted to the infinity has topological degree at most $d^{n-1}$. Hence the critical set of $F$ is not smooth and consequently  such a mapping can never be stable as a mapping from $\Bbb P^n$ to $\Bbb P^n$. 
In particular we can not use here global techniques based on Thom polynomials. 

However, in some cases we can apply local methods using Thom polynomials described by Ohmoto \cite{ohm} (see also \cite{saia}, \cite{saia2}, \cite{ruas}). Indeed, let $F:\C^3\to\C^3$ be a generic mapping. Since the pair $(3,3)$ is a pair of nice dimensions, the mapping $F$ is stable. For $F=(f_1,f_2,f_3)\in\Omega(d_1,d_2,d_3)$ we denote by $\overline{f}_i$ the homogeneous part of $f_i$ of degree $d_i$ and set $F_0=(\overline{f}_1,\overline{f}_2,\overline{f}_3)$. Hence $F_0$ has a stable deformation $F_t(x)=(t^{d_1}f_1(x/t), t^{d_2}f_2(x/t), t^{d_3}f_3(x/t))$. Assume that $(F_0,0)$ is an $\mathcal{A}$-finitely determined germ. Since the deformation $F_t$ contracts all discrete singularities to $0$ as $t\to 0$, we can compute the number of discrete singularities of $F$ using the local formulas of Ohmoto for the mapping $F_0$. Hence the fundamental problem here is to describe $\mathcal{A}$-finitely determined homogeneous mappings $H: \C^3\to\C^3$. We denote by $H(d_1,d_2,d_3)$ the set of all homogeneous polynomial mappings $F=(f_1,f_2,f_3): \C^3\to\C^3$, such that $\deg f_i=d_i$. Our first main result is:
\begin{changednumbertheorem}[\ref{Thm:main1}]
If $\gcd(d_i,d_j)\leq 2$ for $1\leq i<j\leq 3$ and $\gcd(d_1,d_2,d_3)=1$ then there is a non-empty Zariski open subset $U\subset H(d_1,d_2,d_3)$ such that for every mapping $F\in U$ the map germ $(F,0)$ is $\mathcal{A}$-finitely determined.

On the other hand if $\gcd(d_i,d_j)>2$ for some $i,j\in\{1,2,3\}$, $i\neq j$ or $\gcd(d_1,d_2,d_3)>1$, then there are no $\mathcal{A}$-finitely determined homogeneous map germs with degrees $d_1,d_2,d_3$.
\end{changednumbertheorem}

This is an extension of the well-known two-dimensional result of Gaffney-Mond \cite{gm1} to dimension three. This theorem has the following nice application:
\begin{changednumbertheorem}[\ref{Thm:main2}]
If $\gcd(d_i,d_j)\leq 2$ for $1\leq i<j\leq 3$ and $\gcd(d_1,d_2,d_3)=1$ then there is a non-empty Zariski open subset $U_1\subset \Omega(d_1,d_2,d_3)$ such that for every mapping $F\in U_1$ we have:
\begin{itemize}
\item $F$ is stable, in particular the discrete mono- or multi-singularities are of type $A_3$, $A_2A_1$ or $A_1^3$,
\item $F$ has precisely $\# A_3=c_1^3+3c_1c_2+2c_3$ singularities of type $A_3$,
\item $F$ has precisely $\# A_2A_1=(P-3)s_1\# A_2-3\# A_3$ singularities of type $A_2A_1$,
\item $F$ has precisely $\displaystyle \frac{1}{6}\left[(P^2-3P+2)s_1^3-6\# A_2A_1-6\# A_3-3s_1\# A_1^2-4s_1\# A_2 \right]$ singularities of type $A_1^3$.
\end{itemize}

Here $s_1=(d_1+d_2+d_3-3)$, $s_2=(d_1-1)(d_2-1)+(d_1-1)(d_3-1)+(d_2-1)(d_3-1)$, $s_3=(d_1-1)(d_2-1)(d_3-1)$, $P=d_1d_2d_3$, $c_1=s_1$, $c_2=s_2-s_1$,  $c_3=s_3-2s_2+s_1$, $\# A_2=c_1^2+c_2$ and $\# A_1^2=(P-2)s_1^2-2\# A_2$.
\end{changednumbertheorem}

\begin{remark}
The proof works only for $\mathcal{A}$-finite determined map germs, i.e., for $d_1,d_2,d_3$ as in Theorem \ref{Thm:main1}. However, we intend to prove in a separate paper, by using global methods rather then local, that the formula for the number of $A_3$ singularities holds for all degrees. However the formula for the number of $A_2A_1$ and $A_1^3$ singularities depends on $\gcd(d_1,d_2,d_3)$.
\end{remark}

\section{Main result}
For a polynomial mapping $F: \C^n\to\C^m$ let us denote by $C(F)$ the set of critical points of $F$ and by $\Delta(F)=F(C(F))$ the discriminant of $F$.

Moreover, we call a line through the origin a \emph{ray}. We will denote by $(\C^n)^{*t}$ the set $\{(p_1,\ldots,p_t):\ p_i\in\C^n,\ p_i\neq 0$ and $p_i\neq p_j$ for $i\neq j\}$. If $p\in(\C^n)^*$ then we denote by $\C p$ the unique ray passing through $p$.

\begin{lemma}\label{lem_genhom}
Assume that $\gcd(d_i,d_j)\leq 2$ for $1\leq i<j\leq 3$ and $\gcd(d_1,d_2,d_3)=1$. There is a non-empty open subset $U\subset H(d_1,d_2,d_3)$ such that for every mapping $F=(f_1,f_2,f_3)\in U$:
\begin{enumerate}
\item $F^{-1}(0)=\{(0)\}$,
\item if $d_1,d_2,d_3$ are pairwise co-prime then $F$ restricted to any ray contained in $C(F)$ is injective, if they are not co-prime, i.e., $d_i$ is odd and the other two are even, then $F$ restricted to any ray contained in $C(F)\setminus V(f_i)$ is injective and $F$ restricted to any of the finite number of rays contained in $C(F)\cap V(f_i)$ is $2:1$,
\item $F_{|C(F)}$ is injective outside a finite set of rays,
\item if $p\in\Delta(F)$ then $|F^{-1}(p)\cap C(F)|\leq 2$,
\item outside the origin the singularities of $F$ are either folds or cusps, in particular $C(F)\setminus \{0\}$ is smooth,
\item if $F$ has a cusp at $p$ then $F^{-1}(F(p))\cap C(F)=\{p\}$,
\item if $|F^{-1}(p)\cap C(F)|= 2$ then the surface $\Delta(F)$ has a normal crossing at $p$.
\end{enumerate}
\end{lemma}

\begin{proof}
We will consider the sets $X_1,\ldots,X_7\subset(\C^3)^{*t}\times H(d_1,d_2,d_3)$, where $t\in\{1,2,3\}$, consisting of points and mappings that do not satisfy the assertions above. We will show that $\dim(X_1),\ldots,\dim(X_7)\leq \dim(H(d_1,d_2,d_3))$ and consider the projections $X_i\rightarrow H(d_1,d_2,d_3)$. The inequality between dimensions shows, that there is an non-empty open subset $U\subset H(d_1,d_2,d_3)$ over which the fibers of the projections are finite. However, since we consider homogeneous mappings if a point (in $(\C^3)^{*t}$) is in the fiber then the whole ray through this point must also be in the fiber, i.e., the fibers are either empty or infinite. Consequently mappings in $U$ satisfy the desired properties.

The sets $X_i$ will be invariant under linear transformations in the following sense: if $T\in GL(3)$ and $(p_1,\ldots,p_t,F)\in X_i$ then $(T(p_1),\ldots,T(p_t),F\circ T^{-1})\in X_i$. Consequently, to compute $\dim(X_i)$ we will only have to compute the dimensions of selected fibers (in most cases only one) of the projection $X_i\rightarrow (\C^3)^t$.

We denote by $a_{i,j;k}$ the parameters in $H(d_1,d_2,d_3)$ giving the coefficients of $f_k$ at $x^{d_k-i-j}y^iz^j$.

The proofs of all assertions follow the same pattern, thus in later assertions we will omit the details explained in the proofs of earlier ones. When relevant we will first assume that $d_1,d_2,d_3$ are pairwise co-prime and later consider the case when they are not. By symmetry we may assume that when $d_1,d_2,d_3$ are not pairwise co-prime then $d_1$ and $d_2$ are even and $d_3$ is odd.

(1) Consider $X_1=\{(p,F)\in(\C^3)^*\times H(d_1,d_2,d_3): F(p)=(0,0,0)\}$. As explained above we have to show that $\dim(X_1)\leq \dim(H(d_1,d_2,d_3))$ and this follows from the fact that $\dim(X_1\cap \{(1,0,0)\}\times H(d_1,d_2,d_3))\leq \dim(H(d_1,d_2,d_3))-3$. Let $X_1'$ denote $X_1\cap \{(1,0,0)\}\times H(d_1,d_2,d_3)$, we will treat it as a subset of $H(d_1,d_2,d_3)$. We obtain the equations of $X_1'$ by substituting $(1,0,0)$ into the equations of $X_1$, we have $f_1(1,0,0)=\sum a_{i,j;k}x^{d_k-i-j}y^iz^j(1,0,0)=a_{0,0;1}=0$ and $f_2(1,0,0)=a_{0,0;2}=0$ and $f_2(1,0,0)=a_{0,0;3}=0$. Thus $X_1'=V(a_{0,0;1},a_{0,0;2},a_{0,0;3})$ has codimension $3$ in $H(d_1,d_2,d_3)$, as required.

(2) Let $F=(f_1,f_2,f_3)$ and $p\in(\C^3)^*$. Note that if any two of $f_1,f_2,f_3$ are nonzero at $p$, say $f_1(p),f_2(p)\neq 0$, then $F$ restricted to $\C p$ is injective. Indeed, if $q=\lambda p$ and $F(p)=F(q)$, then $f_1(p)=f_1(\lambda p)=\lambda^{d_1}f_1(p)$. Thus $\lambda^{d_1}=1$ and similarly $\lambda^{d_2}=1$, if $\gcd(d_1,d_2)=1$ then it follows that $\lambda=1$. If $\gcd(d_1,d_2)=2$ then $\lambda=1$ or $\lambda=-1$.

Thus we have to show that for a generic $F$ we have $C(F)\cap V(f_i,f_j)=\{0\}$. We show the proof for $f_1$ and $f_2$, the other two pairs follow by symmetry. Consider $X_2=\{(p,F)\in(\C^3)^*\times H(d_1,d_2,d_3): f_1(p)=f_2(p)=J(F)(p)=0\}$. Similarly as in the proof of (1) we define $X_2'=X_2\cap \{(1,0,0)\}\times H(d_1,d_2,d_3)$ and treat it as a subset of $H(d_1,d_2,d_3)$. We have to show, that $X_2'$ has codimension $3$. As for $X_1'$, the first two equations of $X_2'$ are $a_{0,0;1}=0$ and $a_{0,0;2}=0$. The third equation is
$$J(F)(1,0,0)=\det\left[\begin{matrix}d_1a_{0,0;1}&a_{1,0;1}&a_{0,1;1}\\d_2a_{0,0;2}&a_{1,0;2}&a_{0,1;2}\\d_3a_{0,0;3}&a_{1,0;3}&a_{0,1;3}
\end{matrix}\right]=0,$$
after substituting $a_{0,0;1}=a_{0,0;2}=0$ it simplifies to $d_3 a_{0,0;3}(a_{1,0;1}a_{0,1;2}-a_{0,1;1}a_{1,0;2})=0$. The three equations are clearly independent, thus $X_2'$ has codimension $3$ in $H(d_1,d_2,d_3)$, as required.

If $\gcd(d_1,d_2)=2$ then we additionally have to show that there is only a finite number of rays contained in $C(F)\cap V(f_3)$. Consider $X_{2a}=\{(p,F)\in(\C^3)^*\times H(d_1,d_2,d_3): f_3(p)=J(F)(p)=0\}$. Similarly as for $X_2$ we show that $X_{2a}$ has codimension $2$, hence the general fiber of the projection $X_2\rightarrow H(d_1,d_2,d_3)$ has dimension $1$, so it must be a finite union of rays.

(3) Consider $X_3=\{(p_1,p_2,F)\in(\C^3)^{*2}\times H(d_1,d_2,d_3): F(p_1)=F(p_2),\ J(F)(p_1)=J(F)(p_2)=0\}$. Let $X_3'$ be a nonempty fiber of the projection to $(\C^3)^{*2}$. By (2) we may assume that $F$ is injective on rays and consider only fibers over $(p_1,p_2)$ where $p_1$ and $p_2$ are not proportional. Since linear transformations induce isomorphisms of the fibers, we may assume that $(p_1,p_2)=((1,0,0),(0,1,0))$. Thus the equations for $X_3'$ are: $(a_{0,0;1},a_{0,0;2},a_{0,0;3})=(a_{d_1,0;1},a_{d_2,0;2},a_{d_2,0;3})$ and 

$$\det\left[\begin{matrix}d_1a_{0,0;1}&a_{1,0;1}&a_{0,1;1}\\d_2a_{0,0;2}&a_{1,0;2}&a_{0,1;2}\\d_3a_{0,0;3}&a_{1,0;3}&a_{0,1;3}
\end{matrix}\right]=\det\left[\begin{matrix}a_{d_1-1,0;1}&d_1a_{d_1,0;1}&a_{d_1-1,1;1}\\a_{d_2-1,0;2}&d_2a_{d_2,0;2}&a_{d_2-1,1;2}\\
a_{d_3-1,0;3}&d_3a_{d_3,0;3}&a_{d_3-1,1;3}\end{matrix}\right]=0.$$

The first three equations define a linear subspace of codimension $3$, the other two clearly do not have a common factor even after restricting to this subspace, i.e., after substituting $a_{0,0;k}$ for $a_{d_k,0;k}$ in the last equation. Thus $X_3'$ has codimension $5$ in $H(d_1,d_2,d_3)$ and consequently $\dim(X_2)\leq\dim(H(d_1,d_2,d_3))+1$. Note that if $(p_1,p_2,F)\in X_3$ then also $(\lambda p_1,\lambda p_2,F)\in X_3$ for $\lambda\in\C^*$, thus the nonempty fibers of the projection $X_3\rightarrow H(d_1,d_2,d_3)$ are infinite and so there are only finitely many of them.

(4) Consider $$X_4=\{(p_1,p_2,p_3,F)\in(\C^3)^{*3}\times H(d_1,d_2,d_3): F(p_1)=F(p_2)=F(p_3),$$ $$\ J(F)(p_1)=J(F)(p_2)=J(F)(p_3)=0\}.$$ Similarly as above we consider the fibers of the projection $X_4\rightarrow (\C^3)^{*3}$. However now we have to consider more than one case: if $p_1,p_2,p_3$ are not coplanar with the origin then we may assume that $(p_1,p_2,p_3)=((1,0,0),(0,1,0),(0,0,1))$, if $p_1,p_2,p_3$ are coplanar with zero then we can only assume that $(p_1,p_2,p_3)=((1,0,0),(0,1,0),(a,b,0))$ for some $a,b\in\C^*$. We denote the fiber by $X_4'$ in the former case and by $X_4^{ab}$ in the latter. If $\gcd(d_1,d_2)=2$ then we must additionally consider the case when two of the points are opposite. In that case we may assume that $(p_1,p_2,p_3)=((1,0,0),(-1,0,0),(0,1,0))$, we denote the fiber by $X_4^-$.

The equations for $X_4'$ are similar to those of $X_3'$. First we have $(a_{0,0;1},a_{0,0;2},a_{0,0;3})=(a_{d_1,0;1},a_{d_2,0;2},a_{d_3,0;3})=(a_{0,d_1;1},a_{0,d_2;2},a_{0,d_3;3})$ which define a linear subspace of codimension $3$. Then we have three equations given by determinants of a matrix. After restricting to the linear subspace the matrices have a common column: $[d_1a_{0,0;1},d_2a_{0,0;2},d_3a_{0,0;3}]$, but otherwise contain disjoint sets of variables. Thus the equations give a transverse intersection outside $V(a_{0,0;1},a_{0,0;2},a_{0,0;3})$ which itself has codimension $3$. So $X_4'$ has codimension $9$, as required.

For $X_4^{ab}$ we obtain the equations $(a_{0,0;1},a_{0,0;2},a_{0,0;3})=(a_{d_1,0;1},a_{d_2,0;2},a_{d_3,0;3})=\linebreak
(\sum a_{i,0;1}a^{d_1-i}b^i,\sum a_{i,0;2}a^{d_2-i}b^i,\sum a_{i,0;3}a^{d_3-i}b^i)$ which again define a linear subspace, though not as nicely as above. Furthermore we have the two equations with determinants from the definition of $X_3'$ and a third one that is derived from $J(F)(a,b,0)=0$. One can show that the last equation is independent from the previous ones, but in fact we do not need it. Note that the set of triples in $(\C^3)^{*3}$ coplanar with the origin has dimension $8$, so it suffices to show that $X_4^{ab}$ has codimension $8$ in $H(d_1,d_2,d_3)$. This way we obtain a peculiar geometric fact: for a generic $F\in H(d_1,d_2,d_3)$ and $p\in\Delta(F)$ if ${p_1,p_2}\in F^{-1}(p)\cap C(F)$ then none of the points in $F^{-1}(p)$ distinct from $p_1,p_2$ lie in the plane spanned by $p_1,p_2$ and the origin.

For $X_4^-$ we have $p_2=-p_1$ so the equation $F(p_1)=F(p_2)$ reduces to $f_3(p_1)=0$. Furthermore the equations $J(F)(p_1)=0$ and $J(F)(p_2)=0$ are equivalent. Thus $X_4^-$ is given only by $6$ independent equations: $F(p_1)=F(p_3)$, $f_3(p_1)=0$, $J(F)(p_1)=J(F)(p_3)=0$. However the set of points in $(\C^3)^{*3}$ satisfying $p_2=-p_1$ has also dimension $6$.

(5) Consider $X_5=\{(p,F)\in(\C^3)^*\times H(d_1,d_2,d_3): J(F)(p)=J_{1,i}(F)(p)=J_{2,i}(F)(p)=0\}$, where $1\leq i\leq 3$ and $J_{1,i}(F)$ is the determinant of the matrix that we obtain from the Jacobian matrix by replacing the row $\nabla f_i=[\frac{\partial f_i}{\partial x_j}]_{1\leq j\leq 3}$ with the row $\nabla J(F)$ and similarly for $J_{2,i}(F)$ by replacing the row $\nabla f_i$ with the row $\nabla J_{1,i}(F)$. Note that $X_5$ describes the set of pairs $(p,F)$ such that the singularity of $F$ at $p$ is a swallowtail or worse, i.e., is an $A_n$ singularity with $n\geq 3$ or a singularity of corank greater than $1$. Thus the only singularities that are not contained in $X_5$ are folds and cusps. Note that it also automatically excludes the possibility of $C(F)$ having singular points, i.e., points where $J(F)(p)=0$ and $\nabla J(F)(p)=[0,0,0]$. So it suffices to prove that $X_5$ has codimension at least $3$ and this can be done by considering the fiber $X_5'$ over $p_1=(1,0,0)$.

By taking the Laplace expansion of $J(F)(p_1)$ with respect to the second column we obtain $-a_{1,0;1}m_{1;1}+a_{1,0;2}m_{2;1}-a_{1,0;3}m_{3;1}$, where $m_{i;1}$ are the suitable minors, e.g., $m_{1;1}=d_2a_{0,0;2}a_{0,1;3}-d_3a_{0,0;3}a_{0,1;2}$. The formula for $J_{1,1}(F)(p_1)$ is too long to conveniently write down, however it is easy to see that it is the sum of $2a_{2,0;1}m_{1;1}^2$ and a polynomial that does not contain $a_{2,0;1}$. Indeed, the term $a_{2,0;1}$ can only come from $\frac{\partial^2 f_1}{\partial y^2}$ which can by only found in $\frac{\partial J(F)}{\partial y}$ by taking the derivative of $\frac{\partial f_1}{\partial y}$. Similarly, $6a_{3,0;1}m_{1;1}^3$is a summand of $J_{2,1}(F)(p_1)$. Consequently $\det\left(\frac{\partial J(F)(p_1),J_{1,1}(F)(p_1),J_{2,1}(F)(p_1)}{\partial a_{1,0;1},a_{2,0;1}, a_{3,0;1}}\right)=12m_{1;1}^6$, which proves that $X_5'\setminus V(m_{1;1})$ has codimension $3$. We make identical computations for $i\in\{2,3\}$ and computation with $J_{1,1}$, $J_{2,1}$, $a_{1,0;1}$, $a_{2,0;1}$, and $a_{3,0;1}$ replaced with $J_{1,i}$, $J_{2,i}$, $a_{1,0;i}$, $a_{2,0;i}$, and $a_{3,0;i}$, respectively, to obtain that 
$X_5'\setminus V(m_{1;1},m_{2;1},m_{3;1})$ has codimension $3$. The set $V(m_{1;1},m_{2;1},m_{3;1})$ has codimension $2$, it is given by the condition that the first and third column of $J(F)(p_1)$ are proportional, however, we can expand $J(F)(p_1)$ with respect to the third column and obtain $a_{0,1;1}m_{1;2}+a_{0,1;2}m_{2;2}-a_{0,1;3}m_{3;2}$. Proceeding as above we obtain that $X_5'\setminus V(m_{i;2})_{1\leq i\leq 3}$ has codimension $3$, since $V(m_{i;1},m_{i;2})_{1\leq i\leq 3}$ has also codimension $3$ it concludes the proof of (5).

(6) Consider $X_6=\{(p_1,p_2,F)\in(\C^3)^{*2}\times H(d_1,d_2,d_3): F(p_1)=F(p_2),\ J(F)(p_1)=J_{1,i}(F)(p_1)=J(F)(p_2)=0\}$. We have to prove that $X_6$ has codimension $6$. The argument is a mix of the arguments in (3) and (5). As above we focus on the fiber over $(p_1,p_2)=((1,0,0),(0,1,0))$. The equations obtained from $F(p_1)=F(p_2)$ define a linear subspace of codimension $3$. From (5) we obtain that $J(F)(p_1)=J_{1,i}(F)(p_1)=0$ give a space of codimension $2$. And the equation obtained from $J(F)(p_2)$ is independent from the previous ones outside $V(a_{0,0;1},a_{0,0;2},a_{0,0;3})$.

If $\gcd(d_1,d_2)=2$ then we must additionally consider the case $p_2=-p_1$. In this case the equation $F(p_1)=F(p_2)$ reduces to $f_3(p_1)=0$ and the equations $J(F)(p_1)=0$ and $J(F)(p_2)=0$ are equivalent. Thus the fiber of $X_6$ over $(p_1,p_2)$ has codimension $3$, however the space of points in $(\C^3)^*\times(\C^3)^*$ satisfying $p_2=-p_1$ has also codimension $3$, so the sum of fibers of this type has codimension $6$.

(7) Consider $X_7=\{(p_1,p_2,F)\in X_2: F(p_1)=F(p_2),\ d_{p_1}F(\C^3)=d_{p_2}F(\C^3)\}$. Note that since $\Delta(F)$ is a hypersurface either the two branches at $F(p_1)$ intersect transversally or they have equal tangent spaces, which is the condition that we added in the definition of $X_7$. As in (3) we look at the fiber over $(p_1,p_2)=((1,0,0),(0,1,0))$ and obtain the equations $(a_{0,0;1},a_{0,0;2},a_{0,0;3})=(a_{d_1,0;1},a_{d_2,0;2},a_{d_2,0;3})$ and $\rank A\leq 2$, where 
$$A=\left[\begin{matrix}
d_1a_{0,0;1}&a_{1,0;1}&a_{0,1;1}&a_{d_1-1,0;1}&d_1a_{d_1,0;1}&a_{d_1-1,1;1}\\
d_2a_{0,0;2}&a_{1,0;2}&a_{0,1;2}&a_{d_2-1,0;2}&d_2a_{d_2,0;2}&a_{d_2-1,1;2}\\
d_3a_{0,0;3}&a_{1,0;3}&a_{0,1;3}&a_{d_3-1,0;3}&d_3a_{d_3,0;3}&a_{d_3-1,1;3}
\end{matrix}\right].$$

After substituting $(a_{0,0;1},a_{0,0;2},a_{0,0;3})=(a_{d_1,0;1},a_{d_2,0;2},a_{d_2,0;3})$ into $A$ the first and the fifth column become equal, so we may cross the fifth one out without altering the rank. We obtain a $3\times 5$ matrix $A'$ with variables as entries, the condition $\rank A\leq 2$ defines a subset of codimension $3$ (on the open set where a $2\times 2$ minor is nonzero the set is given as the zero set of the three $3\times 3$ minors containing that $2\times 2$ minor). Together with the first three equations we obtain a set of codimension $6$.

If $\gcd(d_1,d_2)=2$ then we must additionally consider the case $(p_1,p_2)=((1,0,0),(-1,0,0))$. We obtain the equations $a_{0,0;3}=0$ and $\rank B\leq 2$, where
$$B=\left[\begin{matrix}
d_1a_{0,0;1}&a_{1,0;1}&a_{0,1;1}&-d_1a_{0,0;1}&-a_{1,0;1}&-a_{0,1;1}\\
d_2a_{0,0;2}&a_{1,0;2}&a_{0,1;2}&-d_2a_{0,0;2}&-a_{1,0;2}&-a_{0,1;2}\\
d_3a_{0,0;3}&a_{1,0;3}&a_{0,1;3}&d_3a_{0,0;3}&a_{1,0;3}&a_{0,1;3}
\end{matrix}\right].$$

After substituting $a_{0,0;3}=0$ and adding columns $1,2,3$ to columns $4,5,6$, respectively we obtain
$$B'=\left[\begin{matrix}
d_1a_{0,0;1}&a_{1,0;1}&a_{0,1;1}&0&0&0\\
d_2a_{0,0;2}&a_{1,0;2}&a_{0,1;2}&0&0&0\\
0&a_{1,0;3}&a_{0,1;3}&0&2a_{1,0;3}&2a_{0,1;3}
\end{matrix}\right].$$

The condition $\rank B'\leq 2$ means that either the first two rows are proportional or $a_{1,0;3}=a_{0,1;3}=0$. Both conditions define a subset of codimension $2$, together with $a_{0,0;3}=0$ we obtain codimension $3$. This is sufficient since the space of pairs $(p_1,p_2)\in(\C^3)^{*2}$ such that $p_2=-p_1$ has dimension $3$. 
\end{proof}

\begin{remark}\label{rem_notcoprime}
Note that Lemma \ref{lem_genhom} (2) fails if $\gcd(d_i,d_j)>2$ for some $i,j\in\{1,2,3\}$, $i\neq j$ or if $\gcd(d_1,d_2,d_3)>1$. Indeed, suppose $\gcd(d_1,d_2)=d>2$, for general $F\in H(d_1,d_2,d_3)$ the set $C(F)\cap V(f_3)$ consists of a finite and nonzero number of rays. If $p$ is an element of such a ray then for $\varepsilon^d=1$ we have $F(\varepsilon p)=F(p)$ and the mapping is actually $d:1$ on that ray. If $\gcd(d_1,d_2,d_3)>1$ then $F$ is not generically one to one on $C(F)$.
\end{remark}

We have the following geometric criterion for finite determinacy of homogeneous map germs (see \cite{wall}):

\begin{theorem}\label{th_geomcrit}
Let $F:(\C^3,0)\rightarrow(\C^3,0)$ be a holomorphic map germ. Then $F$ is $\mathcal{A}$-finitely determined if and only if there is a representative $F:U\subset \C^3\rightarrow V\subset \C^3$ such that
\begin{enumerate}
\item $F^{-1}(0)=\{0\}$,
\item the restriction $F_{|U\setminus\{0\}}: U\setminus\{0\}\to V\setminus\{0\}$ is stable.
\end{enumerate}
\end{theorem}

From Theorem \ref{th_geomcrit}, Lemma \ref{lem_genhom} and Remark \ref{rem_notcoprime} we obtain the following theorem:

\begin{theorem}\label{Thm:main1}
If $\gcd(d_i,d_j)\leq 2$ for $1\leq i<j\leq 3$ and $\gcd(d_1,d_2,d_3)=1$ then there is a non-empty Zariski open subset $U\subset H(d_1,d_2,d_3)$ such that for every mapping $F\in U$ the map germ $(F,0)$ is $\mathcal{A}$-finitely determined.

On the other hand if $\gcd(d_i,d_j)>2$ for $1\leq i<j\leq 3$ or $\gcd(d_1,d_2,d_3)>1$, then there are no $\mathcal{A}$-finitely determined homogeneous map germs with degrees $d_1,d_2,d_3$.
\end{theorem}
\begin{proof}
By Lemma \ref{lem_genhom} any $F\in U$ is locally stable, it is also proper, since $F$ is homogeneous and $F^{-1}(0)=0$. Thus by  \cite{mathII} $F: \C^3\setminus\{0\}\to \C^3\setminus\{0\}$ is stable. By Theorem \ref{th_geomcrit} $(F,0)$ is $\mathcal{A}$-finitely determined.

The last statement follows from  Remark \ref{rem_notcoprime}. 
\end{proof}

\section{Counting singularities}

Mappings from $\C^3$ to $\C^3$ have three types of stable discrete mono- or multi-singularities:
\begin{itemize}
\item $A_3$ -- the swallowtail: $(x,y,z)\mapsto (x,y,z^4+y^2z+xz)$
\item $A_2A_1$ -- intersection of cusp edge and fold surface: $\begin{cases} (x_1,y_1,z_1)\mapsto (x_1,y_1,z_1^3+y_1z_1)
\\ (x_2,y_2,z_2)\mapsto (x_2^2,y_2,z_2)\end{cases}$
\item $A_1^3$ -- triple self-intersection of fold surface: $\begin{cases} (x_1,y_1,z_1)\mapsto (x_1,y_1,z_1^2)
\\ (x_2,y_2,z_2)\mapsto (x_2,y_2^2,z_2)\\(x_3,y_3,z_3)\mapsto (x_3^2,y_3,z_3)\end{cases}$
\end{itemize}

Let us denote $s_1=(d_1+d_2+d_3-3)$, $s_2=(d_1-1)(d_2-1)+(d_1-1)(d_3-1)+(d_2-1)(d_3-1)$, $s_3=(d_1-1)(d_2-1)(d_3-1)$ and $P=d_1d_2d_3$. Furthermore let $c_1=s_1$, $c_2=s_2-s_1$ and $c_3=s_3-2s_2+s_1$. Finally let $\# A_2=c_1^2+c_2$ and $\# A_1^2=(P-2)s_1^2-2\# A_2$. The definitions of $c_1,c_2,c_3$ and $\# A_2$ and $\# (A_1)^2$ have a deeper meaning, the former are related to certain quotient Chern classes, the latter to Thom polynomials. We refer the reader to a paper by Ohmoto \cite{ohm} for the details.

We have the following Theorem:

\begin{theorem}\label{Thm:main2}
If $\gcd(d_i,d_j)\le 2$ for $i\not=j$ and $\gcd(d_1,d_2,d_3)=1$, then there is a non-empty Zariski open subset $U_1\subset \Omega_3(d_1,d_2,d_3)$ such that for every mapping $F\in U$ we have:
\begin{itemize}
\item $F$ is stable, in particular the discrete mono- or multi-singularities are of type $A_3$, $A_2A_1$ or $A_1^3$,
\item $F$ has precisely $\# A_3=c_1^3+3c_1c_2+2c_3$ singularities of type $A_3$,
\item $F$ has precisely $\# A_2A_1=(P-3)s_1\# A_2-3\# A_3$ singularities of type $A_2A_1$,
\item $F$ has precisely $\displaystyle \frac{1}{6}\left[(P^2-3P+2)s_1^3-6\# A_2A_1-6\# A_3-3s_1\# A_1^2-4s_1\# A_2 \right]$ singularities of type $A_1^3$.
\end{itemize}
\end{theorem}

\begin{proof}
For $F=(f_1,f_2,f_3)\in\Omega_3(d_1,d_2,d_3)$ we denote by $\overline{f}_i$ the homogeneous part of $f_i$ of degree $d_i$ and set $F_0=(\overline{f}_1,\overline{f}_2,\overline{f}_3)$. By \cite[Theorem 2.7]{fjr} there is a Zariski open set $V\subset \Omega_3(d_1,d_2,d_3)$ such that for every $F\in V$ is transversal to the Thom-Boardman strata. This determines the types of singularities that $F$ may have: $A_1$, $A_2$ and $A_1^2$ are the non-discrete types and $A_3$, $A_2A_1$ and $A_1^3$ are the discrete types. In particular $F$ is locally stable. We let $U_1=\{F\in V:\ F_0\in U\}$, where $U$ is the open set from Lemma \ref{lem_genhom}. If $F\in U_1$ then $F_0$ is proper, so $F$ is also proper. Since $F$ locally stable and proper, it is also stable. Let $F_t(x,y,z)=(t^{d_1}f_1,t^{d_2}f_2,t^{d_3}f_3)(t^{-1}x,t^{-1}y,t^{-1}z)$, then $F_t$ is a stable deformation of $F_0$. Obviously for all $t\neq 0$ the mappings $F_t$ have the same number of singularities, furthermore all the singularities tend to zero when $t$ tends to zero. Thus by \cite[Example 5.9]{ohm} $F$ has the numbers of singularities as written above.
\end{proof}

\end{document}